\documentclass[]{amsart}
\usepackage{amscd,amsthm,amssymb,amsfonts,amsmath,euscript}

\theoremstyle{plain}
\newtheorem{thm}{Theorem}[section]
\newtheorem{lemma}[thm]{Lemma}
\newtheorem{prop}[thm]{Proposition}
\newtheorem{cor}[thm]{Corollary}

\theoremstyle{definition}
\newtheorem{defn}[thm]{Definition}

\theoremstyle{remark}

\newcommand{\nc}{\newcommand}

\def\makeop#1{\expandafter\def\csname#1\endcsname
  {\mathop{\rm #1}\nolimits}\ignorespaces}
\makeop{Hom}   \makeop{End}   \makeop{Aut}   \makeop{Isom}  \makeop{Pic} 
\makeop{Gal}   \makeop{ord}   \makeop{Char}  \makeop{Div}   \makeop{Lie} 
\makeop{PGL}   \makeop{Corr}  \makeop{PSL}   \makeop{sgn}   \makeop{Spf}
\makeop{Spec}  \makeop{Tr}    \makeop{Nr}    \makeop{Fr}    \makeop{disc}
\makeop{Proj}  \makeop{supp}  \makeop{ker}   \makeop{im}    \makeop{dom}
\makeop{coker} \makeop{Stab}  \makeop{SO}    \makeop{SL}    \makeop{SL}
\makeop{Cl}    \makeop{cond}  \makeop{Br}    \makeop{inv}   \makeop{rank}
\makeop{id}    \makeop{Fil}   \makeop{Frac}  \makeop{GL}    \makeop{SU}
\makeop{Trd}   \makeop{Sp}    \makeop{Tr}    \makeop{Trd}   \makeop{diag}
\makeop{Res}   \makeop{ind}   \makeop{depth} \makeop{Tr}    \makeop{st}
\makeop{Ad}    \makeop{Int}   \makeop{tr}    \makeop{Sym}   \makeop{can}
\makeop{length}\makeop{SO}    \makeop{torsion} \makeop{GSp} \makeop{Ker}
\makeop{Adm}   \makeop{lcm}

\def\makebb#1{\expandafter\def
  \csname bb#1\endcsname{{\mathbb{#1}}}\ignorespaces}
\def\makebf#1{\expandafter\def\csname bf#1\endcsname{{\bf
      #1}}\ignorespaces} 
\def\makegr#1{\expandafter\def
  \csname gr#1\endcsname{{\mathfrak{#1}}}\ignorespaces}
\def\makescr#1{\expandafter\def
  \csname scr#1\endcsname{{\EuScript{#1}}}\ignorespaces}
\def\makecal#1{\expandafter\def\csname cal#1\endcsname{{\mathcal
      #1}}\ignorespaces} 

\def\doLetters#1{#1A #1B #1C #1D #1E #1F #1G #1H #1I #1J #1K #1L #1M
                 #1N #1O #1P #1Q #1R #1S #1T #1U #1V #1W #1X #1Y #1Z}
\def\doletters#1{#1a #1b #1c #1d #1e #1f #1g #1h #1i #1j #1k #1l #1m
                 #1n #1o #1p #1q #1r #1s #1t #1u #1v #1w #1x #1y #1z}
\doLetters\makebb   \doLetters\makecal  \doLetters\makebf
\doLetters\makescr 
\doletters\makebf   \doLetters\makegr   \doletters\makegr
     \def\qed{\qedmark\medbreak}%
\def\qedmark{{\enspace\vrule height 6pt width 5pt depth 1.5pt}}%

\normalsize

\makeop{Bl}

\newcommand{\Z}{\mathbb Z}
\newcommand{\Q}{\mathbb Q}

\newcommand{\F}{\mathbb F}

\newcommand{\npr}{\noindent }

\newcommand{\isoto}{\stackrel{\sim}{\to}}
\nc{\embed}{\hookrightarrow}

\newcommand{\ch}{characteristic }

\nc{\ol}{\overline}
\nc{\wt}{\widetilde}
\nc{\opp}{\mathrm{opp}}

\makeop{Ram}
\makeop{Rep}

\begin{document}
\renewcommand{\thefootnote}{\fnsymbol{footnote}}
\setcounter{footnote}{-1}
\numberwithin{equation}{section}


\title[Field embeddings]
  {Embeddings of fields into simple algebras: generalizations and
  applications}  
\author{Chia-Fu Yu}
\address{
Institute of Mathematics, Academia Sinica and NCTS (Taipei Office)\\
6th Floor, Astronomy Mathematics Building \\
No. 1, Roosevelt Rd. Sec. 4 \\ 
Taipei, Taiwan, 10617} 
\email{chiafu@math.sinica.edu.tw}

\date{\today}

\def\d{{\rm d}}
\def\c{{\rm c}}
\def\i{{\rm i}}
\def\falg{\text{$F$-alg}}
\def\Mat{{\rm Mat}}
\def\Mod{{\rm Mod}}

\begin{abstract}
For two semi-simple algebras $A$ and $B$ over an arbitrary ground
field $F$, we give a numerical criterion when $\Hom_\falg(A,B)$, 
the set of $F$-algebra homomorphisms between them, is non-empty. 
We also determine when the orbit set $B^\times \backslash
\Hom_\falg(A,B)$ is finite and give an explicit 
formula for its cardinality. A few applications of main results are
given.

\end{abstract} 

\maketitle



\section{Introduction}
\label{sec:01}
\def\rad{\rm rad}
Semi-simple algebras are the most fundamental objects in
the non-commutative ring theory. They also serve as a useful tool
for studying some basic objects in algebraic geometry and number theory
such as abelian varieties, Drinfeld modules, or elliptic sheaves those
form a semi-simple category. 
While studying objects $X$ as above with extra symmetries, one
considers objects $X$ together with an additional structure 
$\iota:A\to \End(X)\otimes \Q$ on $X$ by a semi-simple algebra $A$. 
This leads to a general question when 
a homomorphism exists from a given semi-simple algebra to another one.

We let $F$ denote the ground field. 
Suppose $A$ is a (finite-dimensional) central simple algebra over 
$F$, and $K$ a finite field extension of $F$. 
In what condition does there exist an $F$-algebra embedding 
of $K$ into $A$? 
The following well-known result answers 
this question; see \cite[13.3 Theorem and Corollary, p.~241]{pierce}.

\begin{thm}\label{11}
  Assume that $[K:F]=\sqrt{[A:F]}$. Then there is 
  an embedding of $K$ into $A$ as $F$-algebras if and only if $K$
  splits $A$, i.e. $A\otimes_F K$ is a matrix algebra over $K$. 
  In this case, $K$ is isomorphic to a strictly maximal subfield of $A$.
\end{thm}

Note that any $F$-algebra homomorphism from $K$ to $A$ is 
an embedding.
In this paper we consider the following general problem of embedding
of a semi-simple algebra into another one: \\

\npr {\bf (P1) } Let $A$ and $B$ be two semi-simple $F$-algebras.
Find a necessary and sufficient condition for them such that there is
an embedding or a homomorphism of $F$-algebras from $A$ into $B$. \\

Let $\Hom_{\text{$F$-alg}}(A,B)$ denote the set of all $F$-algebra
homomorphisms from $A$ into $B$,  and let $\Hom^*_{\falg}(A,B)\subset
\Hom_{\text{$F$-alg}}(A,B)$ 
denote the subset consisting of embeddings.
Problem {\bf (P1)} asks when the set $\Hom_{\falg}(A,B)$ or
$\Hom_{\falg}^*(A,B)$ is non-empty. For an $F$-algebra $A$, denote by
$\rad(A)$ the Jacobson radical of $A$ and $A^{ss}:=A/\rad(A)$ the
maximal semi-simple quotient of $A$. Denote by $A^{\rm o}$ the
opposite algebra of $A$. If $A$ is semi-simple, then one can write
\[ A\simeq \prod_{i=1}^s \Mat_{n_i}(D_i), \]
where $D_i$ are division algebras over $F$, which is also called the
Wedderburn decomposition of $A$. 

One of our main theorems is the following, 
which solves Problem {\bf (P1)}.

\begin{thm}[Theorem~\ref{27}]\label{12} 
Let $A$ and $B$ be two semi-simple $F$-algebras.
  We realize $B$ as $\prod_{j=1}^r \End_{\Delta_j}(V_j)$, where
  $\Delta_j$ is a division $F$-algebra and $V_j$ is a right
  $\Delta_j$-module for each $j$. Write $A=\prod_{i=1}^s A_i$ into the
  product of simple $F$-algebras.  
\begin{enumerate}
  \item For each $j$, write the maximal semi-simple quotient 
\[ (\Delta_j\otimes_F A^{\rm o})^{ss}=\prod_{k=1}^{t_j}
\Mat_{m_{jk}}(D_{jk}) \]
of $\Delta_j\otimes_F A^{\rm o}$ as the product of simple factors (the
Wedderburn decomposition). 
Then the set $\Hom_{\falg}(A,B)$ 
  is non-empty if and only if there are
  non-negative integers $x_{jk}$ for $j=1,\dots, r$ and $k=1,\dots, t_j$
  such that
  \begin{enumerate}
  \item $\sum_{k=1}^{t_j}x_{jk}=\dim_{\Delta_j} V_j$ for all $j$, and
  \item for all $j, k$, one has
 \[ \frac{m_{jk}[D_{jk}:F]}{[\Delta_j:F]} {\Big |} \, x_{jk}. \]
  \end{enumerate}

\item For each $j, i$, write the maximal semi-simple quotient 
\[ (\Delta_j\otimes_F A_i^{\rm o})^{ss}=\prod_{k=1}^{t_{ji}}
\Mat_{m_{jik}}(D_{jik}) \]
of $\Delta_j\otimes_F A_i^{\rm o}$ as the product of simple factors.
Then the set $\Hom^*_{\falg}(A,B)$ 
  is non-empty if and only if there are
  non-negative integers $x_{jik}$ for $j=1,\dots, r$, $i=1,\dots, s$ and 
  $k=1,\dots, t_{ji}$
  such that
  \begin{enumerate}
  \item $\sum_{i,k}x_{jik}=\dim_{\Delta_j} V_j$ for all $j$, 
  \item for all $j, i, k$, one has
 \[ \frac{m_{jik}[D_{jik}:F]}{[\Delta_j:F]} {\Big |}\, x_{jik}, \quad
 \text{and}\] 
   \item for all $i$, the sum $\sum_{j,k} x_{jik}$ is positive. 
  \end{enumerate}
\end{enumerate}
\end{thm}

 



Theorem~\ref{12} provides a numerical criterion for the problem
 {\bf (P1)}. However, it still looks technical.
 In the special case when $B$ is a central simple algebra
 over $F$, the criterion in Theorem~\ref{12} can be simplified
 as follows (This form will be used for applications):

\begin{thm}[Theorem~\ref{29}]\label{125}
  Let $B=\Mat_n(\Delta)$ be a central simple algebra of $F$, where
  $\Delta$ is the division part of $B$. Let $A=\prod_{i=1}^s A_i$ be a
  semi-simple $F$-algebra and let $K_i$ be the center of $A_i$ for each
  $i$. Then there is an
  embedding from the $F$-algebra $A$ into $B$ if and only if there
  are positive integers $n_i$ for $i=1,\dots, s$ such that

\begin{equation}\label{eq:105}
  n=\sum_{i=1}^s n_i, \quad \text{and}
  \quad [A_i:F]\,\mid\, n_i c_i, \quad
  \forall\, i=1,\dots, s,
\end{equation}
where $c_i$ the capacity of the central
simple algebra $\Delta\otimes_F {A_i}^{\rm o}$ over $K_i$.
\end{thm}

Recall (\cite[p.~179]{reiner:mo}, also see Definition~\ref{28}) 
that the {\it capacity} 
 of a central simple algebra $A$ is the number 
$c$ so that $A\simeq \Mat_c(D)$, where $D$ is a division algebra,
also called the division part of $A$. 

The problem {\bf (P1)} does not seem to be studied in this generality
in the  literature. The only related result we know was obtained by 
Chuard-Koulmann and Morales \cite{chuard-koulmann-morales}, 
who studied, among other
things, the embeddability of a Frobenius algebra $A$ into a central
simple algebra $B$ with the condition $[A:F]=\deg(B)$. 
The overlap with \cite{chuard-koulmann-morales} 
is a consequence (Corollary~\ref{210}) of Theorem~\ref{125} 
where $A$ is a commutative etale $F$-algebra and $[A:F]=\deg(B)$ (and
$B$ is still central simple). Two approaches of studying embeddings of
algebras are different. As we also need to deal with 
non-separable algebras, the method of
 Galois cohomology is not applicable in our
situation. \\ 


The second part of this paper studies how many ``essentially
different'' $F$-algebra homomorphisms there are from $A$ into
$B$. 
There are two natural notions of equivalence relations:


\begin{enumerate}
\item Two $F$-algebra homomorphisms $\varphi_1, \varphi_2:A\to B$ are
  said to be {\it equivalent} if there is an element $b\in B^\times$
  such 
  that $\varphi_2=\Int(b)\circ \varphi_1$. That is, 
  $\varphi_2(a)=b\, \varphi_1(a) \,b^{-1}$ for all $a\in A$.
\item Two $F$-algebra homomorphism $\varphi_1, \varphi_2:A\to B$ are
  said to be {\it weakly equivalent} if there is an $F$-automorphism 
  $\alpha\in \Aut_{F}(B)$ of $B$  such
  that $\varphi_2=\alpha \circ \varphi_1$.  
\end{enumerate}

So the problem simply becomes asking for the size of the orbit set 
$B^\times \backslash \Hom_{\falg}(A,B)$ or the orbit set 
$\Aut_F(B)\backslash \Hom_{\falg}(A,B)$, where the groups 
$B^\times$ and $\Aut_F(B)$ act naturally on the set
$\Hom_{\falg}(A,B)$ from the left. 
In this paper we consider the equivalence relation defined
in (1). 
Put
\[ \calO_{A,B}:=B^\times \backslash \Hom_{\falg}(A,B). \]


\npr {\bf (P2) } Is the orbit set 
$\calO_{A,B}$ finite?
If so, then what is its cardinality  
$|\calO_{A,B}|$? \\

Results toward this direction may be viewed as the generalization of
the Noether-Skolem theorem. 
In \cite{pop:ns} F.~Pop and H.~Pop showed that the orbit set
$\calO_{A,B}$ is finite when $B$ is separable over $F$. 
Recall \cite[p.~101]{reiner:mo} that an $F$-algebra is said to be
  {\it separable over $F$} if it is semi-simple and its center $Z$ is
  etale over $F$, i.e. $Z$ is a finite product of finite separable
  field extensions of $F$. 
  The second main result of this paper solves the problem 
  {\bf (P2)} completely. The answer turns out to be negative in
  general; the finiteness of the orbit set is encoded in the K\"ahler
  modules of the center of $A\otimes_F B$. 
   We give a necessary and sufficient condition
  for $A$ and $B$ so that the orbit set $\calO_{A,B}$ is
  finite. 
  In addition, we compute the precise number of
  $|\calO_{A,B}|$ when $A$ or $B$ is separable over $F$ (a upper bound
  for $|\calO_{A,B}|$ is given in \cite{pop:ns} when $B$ is separable). 
  


We now state the explicit result for $\calO_{A,B}$. We may assume that
$B$ is 
simple; indeed if $B=\prod_{j=1}^r B_j$ is semi-simple, where $B_j$'s
are simple factors, then one has 
$\calO_{A,B}=\prod_{j=1}^r \calO_{A, B_j}$. 
Write $B=\Mat_n(\Delta)$, where $\Delta$ is a division
algebra over $F$. Write the maximal semi-simple quotient of 
$\Delta\otimes_F A^{\rm o}$ into the product of simple factors:  
\begin{equation}
  \label{eq:11}
  (\Delta\otimes_F A^{\rm o})^{ss}\simeq \prod_{i=1}^t \Mat_{m_i}(D_i),
\end{equation}
where $D_i$ is a division algebra. We show (Proposition~\ref{24}) that 
the algebra $\Delta\otimes_F A^{\rm o}$ has the following form
\begin{equation}
  \label{eq:12}
  \Delta\otimes_F A^{\rm o}\simeq \prod_{i=1}^t \Mat_{m_i}(\wt D_i),
\end{equation}
where $\wt D_i$ is an Artinian $F$-algebra whose 
maximal semi-simple quotient is equal to $D_i$. 
Put  $ R_i:=Z(\wt D_i)$ and $Z_i:=Z(D_i)=(R_i)^{ss}$. 
Let $\grm_{R_i}$ be the maximal ideal of $R_i$; one has
$Z_i=R_i/\grm_{R_i}$.   
Since there is an embedding from $\Delta$ into $\Mat_{m_i}(D_i)$, one
has  
$[\Delta:F]\,\mid\, m_i [D_i:F]$. Put
\begin{equation}
  \label{eq:13}
  \ell_i:=m_i [D_i:F]/[\Delta:F]\in \bbN,
\end{equation}
and 
\begin{equation}
  \label{eq:14}
  P(A,B):=\{(x_1,\dots, x_t)\in \Z^t_{\ge 0}\,\mid\, \dim_\Delta
V=\sum_{i=1}^t \ell_i x_i\, \}.
\end{equation}

\begin{thm}[Theorem~\ref{36}]\label{13} Let $A$ be a semi-simple
  $F$-algebra and $B=\Mat_n(\Delta)$ a 
  simple $F$-algebra.
  Let $\wt D_i$, $D_i$, $R_i$, $Z_i$ and $P(A,B)$ be as above. 
  
\begin{enumerate}
\item The orbit set $\calO_{A,B}$ is infinite if and only if there is
  an element $(x_1,\dots, x_t)\in P(A,B)$ such that 
  \begin{equation}
    \label{eq:15}
    \dim_{Z_i} \grm_{R_i}/\grm_{R_i}^2\ge 2\quad\text{and}\quad x_i\ge
    2  
  \end{equation}
for some $i\in \{1,\dots, t\}$.
\item Suppose the orbit set $\calO_{A,B}$ is finite, that is, for
  every element $(x_1,\dots, x_t)\in P(A,B)$, one has either
  $\dim_{Z_i} \grm_{R_i}/\grm_{R_i}^2\le 1$ or $x_i\le 1$ for all
  $i\in \{1,\dots, t\}$.
  Then we have the formula
  \begin{equation*}
   |\calO_{A,B}|=\sum_{(x_1,\dots, x_t)\in P(A,B)} \prod_{i=1}^t
|\Mod(\wt D_i, x_i)|,  
  \end{equation*} 
where $\Mod(\wt D_i,x_i)$ is the set of isomorphism classes of 
$\wt D_i$-modules $W$ with $\length W=x_i$. 
The cardinality $|\Mod(\wt D_i, x_i)|$ is given by the following
formula:
\begin{equation}
  \label{eq:16}
  |\Mod(\wt D_i, x_i)|=
  \begin{cases}
    1 & \text{if $x_i\le 1$}, \\
    p(x_i,e_i) & \text{if $x_i>1$ and $\dim_{Z_i}
    \grm_{R_i}/\grm_{R_i}^2=1$}
  \end{cases}
\end{equation}
where $e_i$ is the smallest positive integer such that
$\grm_{R_i}^{e_i}=0$, and 
$p(x,e)$ denotes of number of all partitions $x=c_1+\dots +c_s$
of $x$ with each part $c_i\le e$. 
\end{enumerate}
\end{thm}






In the remaining part of this paper we give a few applications of 
main results.
We analyze the problem of the local to global principle for 
the algebra embeddings over global fields. 
More precisely, let
$A$ be a central simple algebra over a global field $F$ and $K$ a
finite field extension of $F$ with degree $k=[K:F]$ dividing
$\deg(A)$, the degree of $A$. Can the problem 
of the embedding of $K$ into $A$ 
(i.e.  $\Hom_{\falg}(K,A)\neq \emptyset$) be checked locally? 
It turns out that there are many examples
so that the Hasse principle for embedding $K$ in $A$ does not
hold. That is, the condition 
$\Hom(K\otimes_F F_v, A\otimes_F F_v)\neq \emptyset$
for all places $v$ may not imply that 
$\Hom_{\falg}(K, A)\neq \emptyset$.   

\begin{prop}\label{15}
  Let $K$ be any finite separable field extension of $F$ of degree
  $k>1$. Let $\delta=p_1^{n_1}\cdots p_r^{n_r}$ be a positive integer
  divisible by at least two primes, i.e. $r\ge 2$, with
  $k\,\mid\,\delta$. Assume that $k\le \delta/p_i^{n_i}$ for all $i=1,\dots,
  r$. Then there is a central division algebra $\Delta$ over $F$ of
  degree $\delta$ such
  that the local-to-global principle for an embedding of $K$ in
  $\Delta$ fails.  
\end{prop}

See \S~\ref{sec:44} for the construction of such central division
algebras $\Delta$. This is the first example for the failure of the
Hasse principle of embeddings fields in central simple algebras. 
Furthermore, we give a necessary and sufficient
condition for a pair $(K,A)$ so that the Hasse principle in question
holds. We associate to each pair $(K,A)$ an element 
\[ \bar \bfx=(\bar \bfx_w)_{w\in V^K} \in \bigoplus_{w\in V^K} \Q/\Z \]
 as follows, where $V^K$ and $V^F$ denote the set of all
places of $K$ and $F$, respectively. Put 
\[ \bfx_w:=\frac{\c(A_v)\cdot \gcd(k_w, d_v)}{[K:F]}\in \Q_{>0} \] 
and let $\bar \bfx_w$ be the class of $\bfx_w$ in $\Q/\Z$, where
\begin{itemize}
\item $A_v:=A\otimes_F F_v$ and 
$\c(A_v)$ denotes the capacity of the central simple algebra
$A_v$, where $v$ is the place of $F$ below $w$,
\item  $K_w$ is the
completion of $K$ at $w$ and $k_w:=[K_w:F_v]$, and
\item  $d_v$ is the index of the algebra $A_v$.
\end{itemize}

\begin{thm}[]\label{16} Notations as above. Then there
  is an embedding 
  of $K$ into $A$ over $F$ if and only if there is an embedding of 
  $K_v:=K \otimes_F F_v$ into $A_v$ over $F_v$ for all $v\in V^F$ and
  the element $\bar \bfx$ vanishes. 
\end{thm}

Theorem~\ref{16} states that the element $\bar x$ is the only
obstruction for the Hasse principle for an embedding of $K$ in $A$. 
This can be used to study  the Hasse principle 
for {\it families} of embeddings of fields in central simple algebras. For
example, if $\delta=p^n$ for $n=1,2,3$, where $p$ is
a prime number, then for any central simple algebra $A$ of
degree $\delta$ and any finite field extension $K$ with
$[K:F]\mid\delta$, then the Hasse principle in question for $K$ and $A$
holds. One can also show that this is optimal. That is, if $\delta$ is
not of this form, then there exist a central simple algebra $A$ of
degree $\delta$ and a finite separable field extension $K$ with
$[K:F]\mid \delta$ such that the corresponding Hasse principle fails. 
We refer the reader to \cite{shih-yang-yu} for details. 

Another
applications are the determination of endomorphism algebras of abelian
surfaces with quaternion algebras and the determination of
characteristic polynomials of central simple algebras; see Section for
details. \\



The paper is organized as follows. In Section 2 we determine when 
there is an algebra homomorphism between two given semi-simple
algebras over a field in terms of numerical invariants. 
Section 3 treats the set of equivalence classes
of homomorphisms. We give a necessary and sufficient condition when
this set is finite, and when this set is finite, we determine
precisely its cardinality. 
In Section 4, we give a few applications of general reults 
obtained in Sections 2 and 3.




 


\section{The existence of $F$-algebra homomorphisms}
\label{sec:02}

\subsection{Setting}
\label{sec:21}


Let $F$ denote the ground field, which is arbitrary in this and next sections.
All $F$-algebras in this paper 
are assumed to be finite-dimensional as
$F$-vector spaces. As the standard convention, an $F$-algebra
homomorphism between two $F$-algebras is a ring homomorphism
over $F$, which particularly sends the identity to the identity.  
For the convenience of discussion, we introduce the following basic
notion.

\begin{defn}\label{21}
  Let $V$ be a finite-dimensional vector space 
  over $F$, and
  $A$ an (finite-dimensional) arbitrary $F$-algebra. We say that {\it
  $V$ is $A$-modulable} if there is a right (or left) $A$-module structure on
  $V$. 
  If $B$ is any $F$-subalgebra of $A$ and $V$ is already a right
  (resp. left)
  $B$-module, then by saying $V$ is $A$-modulable we mean that the right
  (resp. left) $A$-module structure on $V$ is required to be 
   compatible with the underlying $B$-module structure on $V$.    
\end{defn}

\begin{thm}\label{22} 
  Let $A$ and $B$ be two semi-simple $F$-algebras.
  We realize $B$ as $\prod_{j=1}^r \End_{\Delta_j}(V_j)$, where
  $\Delta_j$ is a division $F$-algebra and $V_j$ is a right
  $\Delta_j$-module for each $j$. Write $A=\prod_{i=1}^s A_i$ into simple
  factors as $F$-algebras.  
\begin{enumerate}
  \item The set $\Hom_{\falg}(A,B)$ 
  is non-empty if and only if for each $j=1,\dots, r$, there is a
  decomposition of $V_j$  
\[ V_j=V_{j1}\oplus \dots, \oplus V_{j s} \]
into $\Delta_j$-subspaces such that the $\Delta_j$-vector
  space $V_{ji}$ is $\Delta_j\otimes_F A^{\rm
  o}_i$-modulable for $i=1,\dots, s$, where $A^{\rm o}_i$ denotes the
  opposite algebra of $A_i$. 
\item The set $\Hom^*_{\falg}(A,B)$ is non-empty if and only if
  in addition each direct sum $\oplus_{j=1}^r V_{ji}$ is non-zero for
  $i=1,\dots,s$.  
\end{enumerate}
\end{thm}
\begin{proof}
Suppose we have a homomorphism $\varphi:A\to B$ of $F$-algebras, then
we have a $\Delta_j$-linear action of $A$ on $V_j$ for each $j=1,\dots,
r$. The decomposition $A=\prod_{i=1}^s A_i$ gives a decomposition of
$V_j$
\[ V_j=V_{j1}\oplus \dots \oplus V_{js}, \]
where each $V_{ji}$ is a $(A_i,\Delta_j)$-bimodule or a
right $\Delta_j\otimes_F A^{\rm o}_i$-module. If $\varphi$ is an
embedding, then for each $i$ at least one of $V_{ji}$ 
for $j=1,\dots, r$ is non-zero. Therefore, the direct sum $\oplus_j
V_{ji}$ is non-zero. 
Conversely, suppose that we are given such a decomposition with
these properties. Then we have a $\Delta_j$-linear action of $A$ 
on each
vector space $V_j$; this gives an $F$-algebra homomorphism
$\varphi:A\to B$. 
Moreover, suppose that each direct sum  $\oplus_{j=1}^r V_{ji}$ is
non-zero. Then the map restricted to $A_i$  is injective for each
$i$. Therefore, $\varphi$ is an embedding. This proves
the theorem. \qed
\end{proof}

\subsection{Tensor products of two semi-simple algebras}
\label{sec:22}
According to Theorem~\ref{22} we shall need to determine whether a
vector space is $A\otimes_F B$-modulable for semi-simple algebras $A$
and $B$. When $A$ and $B$ are separable $F$-algebras, the
tensor product $A\otimes_F B$ is again separable 
\cite[7.19 Corollary, p.~101]{reiner:mo}. 
However, the tensor product $A\otimes_F B$ of two semi-simple algebras
$A$ and $B$ needs not to be semi-simple. The structure of $A\otimes_F
B$  seems not be well-studied in the literature when 
both $A$ and $B$ are not separable $F$-algebras.






\begin{prop}\label{24}
Let $A$ and $B$ be two semi-simple $F$-algebras. 
Let 
\[ (A\otimes_F B)^{ss}\simeq \prod_{i=1}^t \Mat_{m_i}(D_i) \]
be the Wedderburn decomposition of the maximal semi-simple quotient 
$(A\otimes_F B)^{ss}$.
Then the $F$-algebra 
$A\otimes_F B$ is isomorphic to $\prod_{i=1}^t
\Mat_{m_i}(\wt D_i)$ for some Artinian $F$-algebras $\wt D_i$ with
$(\wt D_i)^{ss}=D_i$. 
\end{prop}
\begin{proof}
  Write $A=\prod_i A_i$ and $B=\prod_j B_j$ into simple factors. Then
  $A\otimes_F B\simeq \prod_{i,j} A_i\otimes_F B_j$. Therefore,
  without loss of generality we may
  assume that $A$ and $B$ are simple. Let $K$ and $Z$ be the center of
  $A$ and $B$, respectively. Since $K\otimes_F Z$ is a commutative 
  Artinian
  $F$-algebra, it is isomorphic to the product $\prod_{i=1}^t R_i$ of
  some local Artinian  $F$-algebra $R_i$. Each $R_i$ is a both
  $K$-algebra and $Z$-algebra. 
We have
  \begin{equation}
    \label{eq:21}
    A\otimes_F B \simeq A\otimes_K (K \otimes_F Z) \otimes_Z B\simeq
    \prod_{i=1}^t A\otimes_K  R_i \otimes_Z B. 
  \end{equation}
Let $R$ be one of $R_i$, and  let $\grm$ be the maximal ideal of
  $R$ and $Z':=R/\grm=(R)^{ss}$ be the
  residue field.
Put $\calC:=A\otimes_K  R \otimes_Z B$. 
Since $A$ and $B$ are respective central simple $K$ and $Z$-algebras,
  the map $I\mapsto A\otimes I \otimes B$ gives a bijection 
between the set of (2-sided) ideals of $R$ and that of $\calC$. 
In particular, $\rad(\calC)=A\otimes \grm \otimes B$. Put $\ol
{\calC}:=\calC/\rad(\calC)$. One has
\[ \ol {\calC}:=A\otimes_K Z' \otimes_Z B=\Mat_{m}(D), \]
where $D$ is a central division $Z'$-algebra. 

Let $\{\epsilon_{ij}\}\subset \ol \calC$ be the standard basis for
$\Mat_m(D)$ over $D$; one has $\epsilon_{ij} \epsilon_{kl}=\delta_{jk}
\epsilon_{il}$ for $1\le i,j,k,l\le m$. By \cite[(6.19) Theorem,
p.~86]{reiner:mo}, there exist orthogonal idempotents $e_{11}, \dots,
e_{mm}\in \calC$ such that $\bar e_{ii}=\epsilon_{ii}$ for all $i$ and
$1=e_{11}+\dots+e_{mm}$. For $i=1,\dots, m-1$, choose an element $e_{i,i+1}\in
e_{ii}\calC  e_{i+1,i+1}$ with $\bar e_{i,i+1}=\epsilon_{i,i+1}$. 
The right 
multiplication by $e_{i,i+1}$ (denoted by $e_{i,i+1}$ again) 
is a lift of the right multiplication by
$\epsilon_{i,i+1}: \ol \calC \epsilon_{ii}\simeq
\ol \calC \epsilon_{i+1,i+1}$. By 
\cite[(6.18) Theorem, p.~85]{reiner:mo}, the map 
\[ e_{i,i+1}: \calC e_{ii}\simeq \calC e_{i+1,i+1} \]
is a $\calC$-linear isomorphism. Let
$e_{i+1,i}\in e_{i+1,i+1} \calC e_{ii}$ be an element 
such that the right multiplication by $e_{i+1,i}$ is the
inverse map of the map $e_{i,i+1}: \calC e_{ii}\to \calC
e_{i+1,i+1}$. 
One
thus has $e_{i,i+1} e_{i+1,i}=e_{ii}$ and 
$e_{i+1,i} e_{i,i+1}=e_{i+1,i+1}$. For $i<j$, define $e_{ij}:=e_{i,i+1}
\dots e_{j-1,j}$ and $e_{ji}:=e_{j,j-1} \dots e_{i+1,i}$. Put $\wt
D:=e_{11} \calC e_{11}$. As a left $\wt D$-module, $\calC$ is
generated by $\{e_{ij}\}$ and it is easy to check that
$e_{ij} e_{kl}=\delta_{jk} e_{il}$ for $1\le i,j,k,l\le m$. This shows
that $\calC\simeq \Mat_m(\wt D)$. This completes the proof of
the proposition. \qed             
\end{proof}

 

\begin{lemma}\label{25}
  Let $D$ be a division algebra over a field $F$ and 
  $\wt D$ be an Artinian $F$-algebra with 
  maximal semi-simple quotient $(\wt D)^{ss}=D$. Then 
  an $F$-vector space
  $V$ is $\wt D$-modulable if and only if 
\[ [D:F]\mid\dim_F V. \]   
\end{lemma}
\begin{proof}
  If $V$ is a $\wt D$-module, then there exists a filtration of $\wt
  D$-submodules $V^i:=({\rad}(\wt D))^i V$. Each graded piece
  $V^i/V^{i+1}$ is a $D$-module, and hence  
\[ [D:F]\mid\dim_F V^i/V^{i+1}. \]
  Since $\dim_F V=\sum_i \dim_F  V^i/V^{i+1}$, one has $[D:F]\mid\dim_F
  V$.  
  Conversely, if $V$ is a $D$-module, 
  then it is also a $\wt D$-module by the
  inflation. The condition $[D:F]\mid\dim_F V$ implies the a $D$-module
  structure on $V$ exists. This completes the proof of the lemma. \qed
\end{proof}

Now we are ready to prove our first main theorem.

\begin{thm}\label{27}
Let notations be as in Theorem~\ref{22}.
\begin{enumerate}
  \item For each $j$, write the maximal semi-simple quotient 
\[ (\Delta_j\otimes_F A^{\rm o})^{ss}=\prod_{k=1}^{t_j}
\Mat_{m_{jk}}(D_{jk}) \]
of $\Delta_j\otimes_F A^{\rm o}$ as the product of simple factors.
Then the set $\Hom_{\falg}(A,B)$ 
  is non-empty if and only if there are
  non-negative integers $x_{jk}$ for $j=1,\dots, r$ and $k=1,\dots, t_j$
  such that
  \begin{enumerate}
  \item $\sum_{k=1}^{t_j}x_{jk}=\dim_{\Delta_j} V_j$ for all $j$, and
  \item for all $j, k$, one has
 \[ \frac{m_{jk}[D_{jk}:F]}{[\Delta_j:F]} {\Big |}\, x_{jk}. \]
  \end{enumerate}

\item For each $j, i$, write the maximal semi-simple quotient 
\[ (\Delta_j\otimes_F A_i^{\rm o})^{ss}=\prod_{k=1}^{t_{ji}}
\Mat_{m_{jik}}(D_{jik}) \]
of $\Delta_j\otimes_F A_i^{\rm o}$ as the product of simple factors.
Then the set $\Hom^*_{\falg}(A,B)$ 
  is non-empty if and only if there are
  non-negative integers $x_{jik}$ for $j=1,\dots, r$, $i=1,\dots, s$ and 
  $k=1,\dots, t_{ji}$
  such that
  \begin{enumerate}
  \item $\sum_{i,k}x_{jik}=\dim_{\Delta_j} V_j$ for all $j$, 
  \item for all $j, i, k$, one has
 \[ \frac{m_{jik}[D_{jik}:F]}{[\Delta_j:F]} {\Big |}\, x_{jik}, \quad
 \text{and}\] 
   \item for all $i$, the sum $\sum_{j,k} x_{jik}$ is positive. 
  \end{enumerate}
\end{enumerate}

\end{thm}

\begin{proof}
(1) By Proposition~\ref{24}, we have 
\[ (\Delta_j\otimes_F A^{\rm o})=\prod_{k=1}^{t_j}
\Mat_{m_{jk}}(\wt D_{jk}) \]
for some Artinian $F$-algebras $\wt D_{jk}$ with 
$(\wt D_{jk})^{ss}=D_{jk}$. 
By Theorem~\ref{22}, the set
$\Hom_{\falg}(A,B)$ is non-empty if and only if $V_j$ is
$\Delta_j\otimes A^{\rm o}$-modulable for all $j$. In this case there is a
decomposition of $\Delta_j$-submodules of $V_j$, 
\[ V_j=\bigoplus_{k=1}^{t_j} V_{jk}, \] 
such that each
$V_{jk}$ is $\Mat_{m_{jk}}(\wt D_{jk})$-modulable. This is equivalent
to, by Lemma~\ref{25}, that $m_{jk}[D_{jk}:F]\mid\dim_F V_{jk}$. 
Put $x_{jk}:=\dim_{\Delta_j}V_{jk}$. Then the integers $x_{jk}$ 
satisfy the
conditions (a) and (b). 

(2) By Proposition~\ref{24}, for each $i$ and $j$ we have
\[ \Delta_j\otimes A^{\rm o}_i=\prod_{k=1}^{t_{ji}}
\Mat_{m_{jik}}(\wt D_{jik}) \]
for some Artinian $F$-algebras $\wt D_{jik}$ with 
$(\wt D_{jik})^{ss}=D_{jik}$. 
Then as in (1) for each $j$ 
we have a decomposition
$V_j=\oplus_{i,k} V_{jik}$ of $\Delta_j$-submodules so that each $V_{jik}$
is a $\Mat_{m_{jik}}(\wt D_{jik})$-module and their sum $\oplus_k
V_{jik}$ forms a right $\Delta_j\otimes A_i^{\rm o}$-module. Put
$x_{jik}:=\dim_{\Delta_j}(V_{jik})$. As in (1), the integers $x_{jik} $
satisfy the conditions (a) and (b). Furthermore, a homomorphism $\varphi\in
\Hom_{\falg}(A,B)$ is injective if and only if each simple factor $A_i$ acts
faithfully on the linear spaces $\{V_{jik}\}_{j,k}$. The latter is
equivalent 
to the condition $\sum_{j,k} x_{jik}>0$ for all $i=1,\dots, s$. 
This proves the theorem. \qed   
\end{proof}



\subsection{Special cases}
We apply the general theorem (Theorem~\ref{27}) to the special case
where $B$ is a central simple $F$-algebra.
Recall \cite[p.~179, p.~253]{reiner:mo} 
the following definition for central simple algebras.

\begin{defn}\label{28}
The {\it degree}, {\it capacity}, and {\it index} of a
central simple algebra $B$ over $F$ are defined as
\[ \deg(B):=\sqrt{[B:F]},\quad \c(B):=m,\quad \i(B):=\sqrt{[\Delta:F]}, \]
if $B\cong \Mat_m(\Delta)$, where $\Delta$ is a division algebra over
$F$, which is uniquely determined by $B$ up to isomorphism.
The algebra $\Delta$ is also called the {\it division part} of $B$.
\end{defn}

\begin{thm}\label{29}
  Let $B=\Mat_n(\Delta)$ be a central simple algebra of $F$, where
  $\Delta$ is the division part of $B$. Let $A=\prod_{i=1}^s A_i$ be a
  semi-simple $F$-algebra and let $K_i$ be the center of $A_i$ for each
  $i$. Then there is an
  embedding of the $F$-algebra $A$ into $B$ if and only if there
  are positive integers $n_i$ for $i=1,\dots, s$ such that
\begin{equation}\label{eq:22}
  n=\sum_{i=1}^s n_i, \quad \text{and}
  \quad [A_i:F]\,\mid\, n_i c_i, \quad
  \forall\, i=1,\dots, s,
\end{equation}
where $c_i$ the capacity of the central
simple algebra $\Delta\otimes_F A_i^{\rm o}$ over $K_i$.
\end{thm}
\begin{proof}
  Write 
\[ \Delta\otimes_F A_i^{\rm o}=(\Delta\otimes_F K_i) \otimes_{K_i}
A_i^{\rm o}=\Mat_{c_i}(D_i) \]
and we have 
\begin{equation}
  \label{eq:23}
  [\Delta:F][A_i:F]=c_i^2 [D_i:F]. 
\end{equation}
By Theorem~\ref{27}, there exists an embedding of the $F$-algebra $A$ to
$B$ if and only if there are {\it positive} integers $n_i$ for $i=1,\dots,
s$ such that $n=\sum_{i=1}^s n_i$ and
\begin{equation}
  \label{eq:24}
  \frac{c_i[D_i:F]}{[\Delta:F]}\, {\Big |}\, n_i, \quad \forall\,
  i=1,\dots, s.
\end{equation}
Using (\ref{eq:23}), the condition (\ref{eq:24}) is equivalent to
$[A_i:F]\,|\, n_i c_i$. This proves the theorem. \qed  
\end{proof}

We apply Theorem~\ref{29} to the case where the semi-simple
algebra $A=K$ is commutative and obtain the following well-known
result (cf.   
\cite[Proposition 2.6]{prasad-rapinchuk:embed10} 
and \cite[Section 4]{chuard-koulmann-morales}).

\begin{cor}\label{210} 
  Let $B=\Mat_{n}(\Delta)$ be a central simple algebra over $F$ and
  $K=\prod_{i=1}^s  K_i$ is commutative semi-simple
  $F$-algebra. Assume that  $[K:F]=\deg(B)$. Then there exists an
  embedding of $K$ into $A$ if and only if each $K_i$ splits $B$.
\end{cor}
\begin{proof} 
  By Theorem~\ref{29}, the $F$-algebra $K$ can be embedded into $B$ if and only if
  there are positive integers $n_i$ for $i=1,\dots, s$ such that 
\begin{equation}\label{eq:25}
  n=\sum_{i=1}^s n_i, \quad \text{and}
  \quad [K_i:F]\,\mid\, n_i c_i, \quad
  \forall\, i=1,\dots, s,
\end{equation}
where $c_i=\c(\Delta\otimes_F K_i)$. We need to show that
$\deg(\Delta)=c_i$, i.e. $K_i$ splits $\Delta$ for all $i$. 
Since $[K:F]=\deg(B)$, we have
\[ [K:F]=\deg (B) = \sum_i n_i \deg(\Delta) \ge \sum_i n_i c_i \ge \sum_i
  [K_i:F]=[K:F]. \]
It follows that $c_i=\deg(\Delta)$ for each $i$.
\qed
\end{proof}

\section{$F$-algebra homomorphisms}
\label{sec:03}

In this section we shall find a
necessary and sufficient condition for which 
the orbit set 
\begin{equation}
  \label{eq:31}
  \calO_{A,B}:=B^\times \backslash \Hom_\falg(A,B)
\end{equation}
is finite, where $A$ and $B$ are given semi-simple $F$-algebras. Then
we determine the cardinality $|\calO_{A,B}|$ when it is finite.

If we write $B=\prod_{j=1}^r B_j$ into simple factors, then one has
\begin{equation}
  \label{eq:32}
  \calO_{A,B}=\prod_{j=1}^r \calO_{A,B_j}.
\end{equation}
Therefore, we may and do assume that $B$ is
simple. Write $B=\End_{\Delta}(V)$, where $\Delta$ is a division
algebra over $F$. Write 
\begin{equation}
  \label{eq:33}
  \Delta\otimes_F A^{\rm o}\simeq \prod_{i=1}^t \Mat_{m_i}(\wt D_i)
\end{equation}
as in Proposition~\ref{24} and put 
\begin{equation}
  \label{eq:34}
  D_i:=(\wt D_i)^{ss}, \quad R_i:=Z(\wt D_i),\quad
Z_i:=Z(D_i)=(R_i)^{ss}. 
\end{equation}
Since there is an embedding of $\Delta$ into $\Mat_{m_i}(D_i)$, we have 
\[ [\Delta:F]\,\mid\, m_i [D_i:F]. \]
Put
\begin{equation}
  \label{eq:35}
  \ell_i:=m_i [D_i:F]/[\Delta:F],
\end{equation}
and 
\begin{equation}
  \label{eq:36}
  P(A,B):=\{(x_1,\dots, x_t)\in \Z^t_{\ge 0}\,\mid\, \dim_\Delta
V=\sum_{i=1}^t \ell_i x_i\, \}.
\end{equation}
By Theorem~\ref{27}, the orbit set $\calO_{A,B}$ is non-empty if
and only if there is a decomposition $V=\oplus_i V_i$ of
$\Delta$-submodules such that 
\begin{equation}
  \label{eq:37}
  \dim_\Delta V_i= \ell_i x_i, \quad \text{for some non-negative
  integers $x_i$.}
\end{equation}
or equivalently, the partition set $P(A,B)$ is non-empty.

\begin{defn}\label{31}
  Let $D$ be a division algebra over $F$ and $\wt D$ be an Artinian
  $F$-algebra with $\wt D^{ss}=D$.  
  For any non-negative integer $n$,
  denote by $\Mod(\wt D, n)$ the set of equivalent classes of $\wt
  D$-modules $W$ with $\length W= n$. 
\end{defn}
 
Note that $\length W= n$ if and only if $\dim_F W=n [D:F]$
(cf. Lemma~\ref{25}). 


\begin{lemma}\label{32}
Let $\varphi_1,\varphi_2\in \Hom_\falg(A,\End_\Delta(V))$ be two maps.  
Let $W^{(i)}$ be the $(A,\Delta)$-bimodule on $V$ defined through
$\varphi_i$  
for $i=1,2$.  Then $\varphi_1$ and $\varphi_2$ are equivalent 
if and only if $W^{(1)}$
and $W^{(2)}$ are isomorphic as $(A,\Delta)$-bimodules.   
\end{lemma}
\begin{proof}
  If $W^{(1)}$ and $W^{(2)}$ are isomorphic, then there is a
  $\Delta$-linear automorphism $g:V\to V$ such that the
  following diagram 
\[ 
\begin{CD}
  V @>{g}>> V \\
  @VV\varphi_1(a)V  @VV\varphi_2(a)V \\
  V @>{g}>> V \\
\end{CD} \]
commutes for all $a\in A$. Therefore, $\varphi_2=\Int(g)\circ \varphi_1$.
Conversely, if we are given $g$
such that $\varphi_2=\Int(g)\circ \varphi_1$, then the map $g:V\to V$
is an $(A,\Delta)$-linear isomorphisms from $W^{(1)}$ to
$W^{(2)}$. \qed
\end{proof}
Equivalently, the condition in Lemma~\ref{32} says that $W^{(1)}$ and
$W^{(2)}$ are isomorphic as right $\Delta\otimes_F A^{\rm o}$-modules. 

Using Lemma~\ref{32} and the discussion above we obtain the following
result. Note following from (\ref{eq:35}) and (\ref{eq:37}) that 
$V_i$ is a right $\Mat_{m_i}(\wt D_i)$-module with 
\[ \dim_{F} V_i=[D_i:F] m_i x_i. \]
Using the Morita equivalence, the module $V_i=W_i^{\oplus m_i}$
determines uniquely an element $W_i\in \Mod(\wt D_i,x_i)$. 

\begin{thm}\label{33} Notations being as above.
  There is a bijection between the orbit set $\calO_{A,B}$ and 
  \begin{equation}
    \label{eq:38}
   \coprod_{(x_1,\dots, x_t)\in P(A,B)} \prod_{i=1}^t \Mod(\wt D_i,x_i). 
  \end{equation}
\end{thm}

Theorem~\ref{33} tells us that 
\begin{enumerate}
\item the orbit set $\calO_{A,B}$ is non-empty if and only if so is
  the set $P(A,B)$; 
\item the orbit set $\calO_{A,B}$ is finite if and only if each set 
  $\Mod(\wt D_i,x_i)$ is finite for 
  all $i=1,\dots, t$ and for all $(x_1,\dots, x_t)\in P(A,B)$;
\item if $\calO_{A,B}$ is finite, then 
  \begin{equation}
    \label{eq:39}
   |\calO_{A,B}|=\sum_{(x_1,\dots, x_t)\in P(A,B)} \prod_{i=1}^t
|\Mod(\wt D_i, x_i)|. 
  \end{equation}
\end{enumerate}

It remains to determine when a set of the form $\Mod(\wt D,x)$ is
finite, and to compute 
its cardinality  $|\Mod(\wt D,x)|$ if it is so.
By definition, one has $|\Mod(\wt D, 0)|=1$.

\begin{prop}\label{34} Let $D$ be a division algebra over $F$, and
  let $\wt D$ be an Artinian $F$-algebra with 
  $(\wt D)^{ss}=D$. Let $R:=Z(\wt D)$ and $Z:=R/\grm_R$ be the residue
  field of $R$. 
\begin{enumerate}
  \item The set $\Mod(\wt D, 1)$ is finite and $|\Mod(\wt D, 1)|=1$.
  \item If the ground field $F$ is finite, then the set $\Mod(\wt D,
    x)$ is finite.
  \item If $\grm_R=0$, then $\wt D=D$ and $|\Mod(\wt D, x)|=1$. 
  \item Suppose $\dim_Z \grm_R/\grm_R^2=1$. Let $e$ be the smallest
    positive integer such that $\grm_R^e=0$.
Then for any positive integer $x$, 
   the set
    $\Mod(\wt D, x)$ is finite and
  $|\Mod(\wt D, x)|$ is equal to the number $p(x,e)$ of partitions
  $x=c_1+\dots+c_r$, for some $r\in\bbN$, 
  of $x$ with each part $1\le c_i\le e$.   
\end{enumerate}
\end{prop}
\begin{proof}
  (1) It is clear. (2) This follows from the finiteness of 
\[ \Hom_\falg(\wt D,\Mat_c(F))\subset \Hom_{\text{$F$-lin}}(\wt
D,\Mat_c(F))=\Mat_{\dim \wt D \times c^2} (F), \] 
where $c:=[D:F]x$. (3) It is clear.

(4) Choose a generator $\pi\in \grm_R$ so that $\pi\in R\in \wt D$, 
    $\pi^e=0$ and $\wt D/\pi \wt D=D$. Every finite
      $\wt D$-module is 
      isomorphic to 
\[ \wt D/(\pi^{c_1})\oplus
      \wt D/(\pi^{c_2})\oplus \cdots \oplus
      \wt D/(\pi^{c_r}), \quad \text{for some $r\in \bbN$},\] 
where $1\le c_1\le \dots \le c_r$ are integers with each
      $ c_i\le e$. The proof is similar to that for the classification
      of finite $F[\pi]/(\pi^e)$-modules. We leave the details to the
      reader. This completes the proof of the proposition. \qed
\end{proof}
 

\begin{prop}\label{35} Let $\wt D$, $D$, $R$, $Z$, and $\grm_R$ be as
  in Proposition~\ref{34}. If 
  \begin{enumerate}
  \item[(i)] $\dim_Z \grm_R/\grm_R^2\ge 2$,
  \item [(ii)] the ground field $F$ is infinite, and 
  \item [(iii)] the integer $x\ge 2$, 
  \end{enumerate}
then the set $\Mod(\wt D, x)$ is infinite.
\end{prop}
\begin{proof} 
\def\Ann{{\rm Ann}} 
By Cohen's theorem, one can write 
\[ R\simeq Z[x_1, \dots, x_n]=Z[X_1,\dots,X_n]/(f_j(X))_j, 
\quad n\ge 2,  \]
with every equation $f_j(X)\in (X_1,\dots,X_n)^2$. Then $R$ admits a
quotient 
\[ R_1\simeq Z[x_1,x_2]=Z[X_1,X_2]/(X_1^2,X_1X_2,X_2^2). \]
Put $\wt D_1:=\wt D\otimes_R R_1$, which is a quotient of $\wt D$. 
The inflation operation gives the inclusion map 
$\Mod(\wt D_1,x)\subset \Mod(\wt D, x)$. Therefore,  
it suffices to show that
the set $\Mod(\wt D_1,x)$ is infinite. We shall construct infinitely many
non-isomorphic $\wt D_1$-modules $M$ in $\Mod(\wt D_1,x)$. 
View $D$ as a $\wt D_1$-module by inflation. For any element $a\in F$,
put
\[ M_a:=\wt D_1/(x_1+ax_2)\oplus D^{\oplus x-2}. \] 
It is clear that  $\dim_F M_a=[D:F]x$ 
and that the annihilator $\Ann(M_a)=\wt
D_1 (x_1+ax_2)$. For $a,b\in F$, if $M_a\simeq M_b$ then
$\Ann(M_a)=\Ann(M_b)$ and hence $a=b$. This shows that if $F$ is
infinite then there are infinitely many non-isomorphic $\wt D_1$-modules
in $\Mod(\wt D_1,x)$. \qed
\end{proof}

We refine our main theorem
(Theorem~\ref{33}) by Propositions~\ref{34} and \ref{35} as follows.

\begin{thm}\label{36} Let $A$ be a semi-simple $F$-algebra and $B$ a
  simple $F$-algebra.
  Let $\wt D_i$, $D_i$, $R_i$, $Z_i$ and $P(A,B)$ be as above. 
  
\begin{enumerate}
\item The orbit set $\calO_{A,B}$ is infinite if and only if there is
  an element $(x_1,\dots, x_t)\in P(A,B)$ such that 
  \begin{equation}
    \label{eq:310}
    \dim_{Z_i} \grm_{R_i}/\grm_{R_i}^2\ge 2\quad\text{and}\quad x_i\ge
    2  
  \end{equation}
for some $i\in \{1,\dots, t\}$.
\item Suppose the orbit set $\calO_{A,B}$ is finite, that is, for
  every element $(x_1,\dots, x_t)\in P(A,B)$, one has either
  $\dim_{Z_i} \grm_{R_i}/\grm_{R_i}^2\le 1$ or $x_i\le 1$ for all
  $i\in \{1,\dots, t\}$.
  Then 
  \begin{equation*}
   |\calO_{A,B}|=\sum_{(x_1,\dots, x_t)\in P(A,B)} \prod_{i=1}^t
|\Mod(\wt D_i, x_i)|.  
  \end{equation*} 
Moreover, 
\begin{equation}
  \label{eq:311}
  |\Mod(\wt D_i, x_i)|=
  \begin{cases}
    1 & \text{if $x_i\le 1$}, \\
    p(x_i,e_i) & \text{if $x_i>1$ and $\dim_{Z_i}
    \grm_{R_i}/\grm_{R_i}^2=1$}
  \end{cases}
\end{equation}
where $e_i$ is the smallest positive integer such that
$\grm_{R_i}^{e_i}=0$, and 
$p(x,e)$ denotes of number of all partitions $x=c_1+\dots +c_s$
of $x$ with each part $c_i\le e$. 
\end{enumerate}
\end{thm}


Note that the condition $\dim_{Z_i} \grm_{R_i}/\grm_{R_i}^2\ge 2$ in
(\ref{eq:310})  can
occur only when $F$ is infinite. The general case where $B$ is
semi-simple can be reduced to the simple case as Theorem~\ref{36} 
by (\ref{eq:32}). 
 
As an immediate consequence of Theorem~\ref{36}, the following
  result improves the  main results of F.~Pop and H.~Pop in 
  \cite{pop:ns}.
 
\begin{cor}\label{37} 
  Let $A$ and $B$ be semi-simple $F$-algebras. 
  Assume that $A$ or $B$ is separable over $F$. Then the orbit set 
  $\calO_{A,B}$ is finite and $|\calO_{A,B}|=|P(A,B)|$. 
\end{cor}

We have defined the set $P(A,B)$ when $B$ is simple. When $B$ is
semi-simple, if we write $B=\prod_{j=1}^r B_j$ into simple factors, 
then the set $P(A,B)$ is defined as
\[ P(A,B):=\prod_{j=1}^r P(A,B_j). \]

\section{Some applications}
\label{sec:04}

In this section, we give a few applications of the results in the
previous sections. 

\subsection{Characteristic polynomials of central simple algebras}
\label{sec:41}

Let $A$ be a (f.d.) central simple algebra over an arbitrary base field
$F$. Let $A=\End_{\Delta}(V)=\End_n(\Delta)$, where $\Delta$ is a
central division algebra over $F$, $V$ is a right $\Delta$-vector space
of dimension. 
For any element $x\in A=\Mat_n(\Delta)$, the {\it characteristic
    polynomial of $x$} is defined to be the characteristic polynomial of
  the image of $x$ in $\Mat_{nd}(\bar F)$ under a map
\[ A\to A\otimes_F {\bar F} \stackrel{\rho}{\simeq} \Mat_{nd}(\bar
F), \]
where $\bar F$ is an algebraic closure of $F$ and $d$ is the degree of
$\Delta$. This polynomial is independent of
the choice of the isomorphism $\rho$ and it is defined over $F$. We
call a monic polynomial of degree $nd$ is a characteristic polynomial
of $A$ if it is the characteristic polynomial of some element in $A$. 
A natural question is to determine whether a given polynomial is a
characteristic polynomial of $A$. We have the following result. 

\begin{thm}\label{41}
  Let $f(t)\in F[t]$ be a monic polynomial of degree $nd$ and let
  $f(t)=\prod_{i=1}^s p_i(t)^{a_i}$ be the factorization into
  irreducible polynomials. Put $F_i:=F[t]/(p_i(t))$. Then $f(t)$ is a
  characteristic polynomial of $A$ if and only if for all $i=1,\dots,
  s$, one has
  \begin{itemize}
  \item [(a)] $a_i\, \deg p_i(t)=n_i
    d$ for some positive integer $n_i$, and
  \item [(b)] $[F_i:F]\mid n_i\cdot \c(\Delta\otimes_F F_i)$. 
  \end{itemize}
\end{thm}

The idea of the proof is to reduce first the case where $f(t)$ is a
power of an irreducible polynomial. More precisely, one can show that 
$f(t)$ is a characteristic polynomial if and only if   
\begin{itemize}
  \item [(a)] $a_i\, \deg p_i(t)=n_i
    d$ for some positive integer $n_i$, and
  \item [(c)] each $p_i(t)^{a_i}$ is a characteristic polynomial of
    $\Mat_{n_i}(\Delta)$. 
\end{itemize}
Then one can show that the condition (c) is equivalent to that the
field extension $F_i$ can be embedded into $\Mat_{n_i}(\Delta)$. Then
we use Theorem~\ref{29} to show that this is equivalent to (b). The
details can be found in \cite{yu:charpoly}.

We remark that when $F$ is a non-Archimedean local field, the condition (b)
can be replaced by the following simpler condition
\begin{itemize}
\item [(d)] $\deg p_i(t) \mid n_i \gcd(d, \deg p_i(t))$.
\end{itemize}
This follows from the formula 
$\inv_{F_i}(\Delta\otimes_F F_i)=\inv_F(\Delta)[F_i:F]$, where
$\inv_F(\Delta)$ is the invariant of $\Delta$.  

\subsection{Endomorphism algebras of QM abelian surfaces}
\label{sec:42}

We can apply the embedding result to 
determine all possible endomorphism algebras
of abelian surfaces with quaternion multiplication (QM). 
Let $D$ be an indefinite quaternion division algebra 
over the field $\Q$ of rational
numbers. It is interesting to 
find out all $\Q$-algebras $E$ containing $D$ which appear as endomorphism
algebras of abelian surfaces. In other words, we would like to know
which endomorphism algebra appears in the Shimura curve $X_D$
associated to the quaternion algebra $D$ (and with additional
data).
Studying whether or not a semi-simple algebras over $\Q$ can appear as
the endomorphism algebra of an abelian variety is a way to understand
structures of abelian varieties. See Oort \cite{oort:endo} for
detail discussions and extensive information for this problem. 

\begin{thm}\label{42}
  Let $D$ be an indefinite quaternion division algebra over $\Q$, and
  let $A$ be an abelian surface over a field $k$ 
  with quaternion multiplication by $D$,
  i.e. an abelian surface together with a $\Q$-algebra 
  embedding $\iota: D\to E:=\End^0(A):=\End(A)\otimes_{\Z} \Q$. 
  \begin{itemize}
  \item [(1)] Suppose that $A$ is not simple. Then $A$ is isogenous to
    $C^2$ for an elliptic curve $C$ and the algebra 
    $E$ is isomorphic to one of the
    following
    \begin{itemize}
    \item [(i)] $\Mat_2(K)$, where $K$ is any imaginary quadratic field
    which splits $D$,
    or
    \item [(ii)] $\Mat_2(D_{p,\infty})$, where $D_{p,\infty}$ is the
    quaternion algebra over $\Q$ ramified exactly at $\{p,\infty\}$. 
    This occurs if and only if $C$ is a
    supersingular elliptic curve over the base field $k\supset \F_{p^2}$. 
    \end{itemize}
    \item [(2)] Suppose that $A$ is simple. Then 
   \begin{itemize}
  \item[(i)] $E\simeq D$, or
  \item[(ii)] $E\simeq D_K:=D\otimes_\Q K$ 
    for some imaginary quadratic field $K$. In this case, 
    the abelian surface $A$ is in \ch $p>0$ for some prime
    $p$ and it is supersingular. 
  \end{itemize}
  \end{itemize}
\end{thm}

The algebra $D$ (the case (2) (i)) can  occur as we can take a
generic complex abelian surface with QM by $D$. 
Recall that an abelian
variety in \ch $p>0$ is said to be {\it supersingular} if it is
isogenous to a product of supersingular elliptic curves over a finite
field  extension. 
The case (ii) of
Theorem~\ref{11} (2) can occur only when the quaternion algebra $D$
satisfies certain conditions. In this case, the algebra $E$ is
determined by its center $K$, and there are only a finite list 
of possibilities for such $K$. 
More precisely, we have the following result.

\begin{thm}\label{43}
  Let $A$ be a simple supersingular
  abelian surface over a finite field $\F_q$ of \ch $p>0$ with
  quaternion multiplication by $D$. Let $E:=\End^0(A)$ and $S$ be the
  discriminant of $D$. Then  
  \begin{itemize}
  \item [(1)] The center $K$ of $E$ is isomorphic 
    to $\Q(\zeta_n)$ for $n=3, 4$, or $6$.
  \item [(2)] $p\mid S$ and $p \equiv 1 \pmod n$, where $n$ is as
    above, and for any other prime $\ell \mid S$, one has either
    $\ell | n$ or $\ell \equiv -1 \pmod n$, that is, $\ell$ does not
    split in the quadratic field $\Q(\zeta_n)$.
  \item [(3)] $E\simeq D\otimes_\Q K$.   
  \end{itemize}  
\end{thm}

Theorem~\ref{43} states that there are three possibilities for
endomorphism algebras $E$ of simple supersingular abelian surfaces
over finite fields: $E\simeq D\otimes_\Q \Q(\zeta_n)$ for $n=3,4, 6$. 
However, not all of them occur; it depends on the quaternion algebra $D$.
The algebra $D\otimes_\Q \Q(\zeta_n)$ occurs if and only if there is
exactly one prime $p\mid S$ such that $p\equiv 1 \pmod n$. 

These results contribute (Theorems~\ref{42} and \ref{43})
new cases to the problem of semi-simple algebras appearing 
as endomorphism algebras of abelian varieties as mentioned above (see
Oort~\cite{oort:endo}). The proofs and
details will be published elsewhere. \\

In the last part of this section, we add some details to the
local-global principle for embeddings of fields in simple algebras in 
Section~\ref{sec:01}. 

\subsection{Embeddings over global fields}\label{sec:43}
We let the base field $F$ be a global field. 
Let $A$ be a central simple algebra over $F$ and let $K$ be a finite
field extension of $F$ with $k:=[K:F]\mid \deg(A)$. Consider the
local-global principle for embeddings of $K$ in $A$. That is, if there
exists an embedding of $K_v:=K\otimes_F F_v$ in $A_v:=A\otimes_F F_v$
for all places $v\in V^F$, does there exist an embedding of $K$ in
$A$? When $K$ has the maximal degree, the answer is yes and this is a
useful result.  

We use the following notations

\begin{itemize}

\item $A=\End_\Delta(V)$, where $\Delta$ is the division part of $A$,
  and $V$ is a finite right $\Delta$-module of rank $n$.

\item $k:=[K:F]$ and $\delta:=\deg(\Delta)$.

\item For any place $v$ of $F$, denote by $F_v$ the completion of $F$
  at $v$. Put
  \[ K_v:=K\otimes F_v=\prod_{w|v} K_w, \quad A_v:=A\otimes F_v, \quad
  \Delta_v=\Delta\otimes F_v=\Mat_{s_v}(D_v), \]
  where $D_v$ is the division part of the central simple algebra
  $\Delta_v$ and $s_v$ is the capacity of $\Delta_v$.

\item $k_w:=[K_w:F_v]$ and $d_v:=\deg(D_v)$,
  where $w$ is a place of $K$ over $v$.

\item $\Delta\otimes_F K=\Mat_c(\Delta')$ and $\delta':=\deg(\Delta')$,
  where $\Delta'$ is the
  division part of the central simple algebra $\Delta\otimes K$ over
  $K$, and $c$ is its capacity. One has
\begin{equation}
    \label{eq:48}
   \delta=\delta' c.
\end{equation}

\item For any place $w$ of $K$, put
\[ \Delta'_w:=\Delta'\otimes_K K_w=\Mat_{t_w}(D'_w), \quad
d'_w:=\deg(D'_w), \]
where $D'_w$ is the division part of the central simple algebra
$\Delta'_w$ and $t_w$ is the local capacity of $\Delta'$ at $w$.

\item $c_w:=\c(D_v\otimes_{F_v} K_w)$, i.e. $D_v\otimes_{F_v}
  K_w=\Mat_{c_w}(D'_w)$. One has
  \begin{equation}
    \label{eq:49}
    d_v=d_w' c_w.
  \end{equation}
It follows from $\inv(D'_w)=\inv(D_v)[K_w:F_v]$ (see
  \cite{serre:lf}) that
  \begin{equation}
    \label{eq:413}
    c_w=(d_v,k_w).
  \end{equation}

\item For each place $v$ of $F$, write
\[ \inv_v(\Delta)=\frac{a_v}{\delta}=\frac{ a'_v\, s_v}{d_v\,
  s_v}=\frac{a'_v}{d_v},\quad (a'_v,d_v)=1
  \text{\ and\ } s_v=(a_v,\delta). \]
  One has, by the Grunwald-Wang theorem
\begin{equation}\label{eq:411}
  \delta=\lcm \{d_v\}_{v\in V^F} \quad \text{and}\quad
  \left ( \gcd\{a_v\}_{v\in V^F},\delta \right)=1,
\end{equation}
where $V^F$ denotes the set of all places of $F$.
\item For each place $w$ of $K$, write
\[ \inv_w(\Delta')=\frac{b_w}{\delta'}=\frac{{b'}_w\, t_w}{d'_w\,
  t_w}=\frac{{b'}_w}{d'_w},\quad ({b'}_w,d'_w)=1
  \text{\ and\ } t_w=(b_w,\delta'). \]
  One has
\begin{equation}\label{eq:412}
  \delta'=\lcm \{d'_w\}_{w\in V^K} \quad \text{and}\quad
  \left ( \gcd\{b_w\}_{w\in V^K},\delta' \right)=1,
\end{equation}
where $V^K$ denotes the set of all places of $K$.
\end{itemize} \

Given $K$ and $A$, we have, for each place $v$ of $F$,
\begin{itemize}
\item a tuple
$(k_w)_{w|v}$ of positive integers,  and
\item a rational number $\inv_v(\Delta)=a'_v/d_v$
\end{itemize}
satisfying the following conditions:
\begin{itemize}
\item [(a)] $\sum_{w|v} k_w=k$ for all $v\in V^F$,
\item [(b)]
  \begin{itemize}
  \item [(i)] $d_v=1$ if $v$ is a complex place,
  \item [(ii)]$d_v\in\{1,2\}$ if $v$
  is a real place,
  \item [(iii)] $d_v=1$ for almost all $v$, and
  \item [(iv)] (Global class field theory) one has
\[  \sum_{v\in V^F} \frac{{a'}_v}{d_v}=0. \]
  \end{itemize} 
\end{itemize}

We compute all other numerical invariants $\delta$, $c_w$, $d'_w$,
$\delta'$ and $c$ as follows.
\begin{itemize}
\item [(i)] The (global) degree $\delta$ of $\Delta$ can be
computed by (\ref{eq:411}).

\item [(ii)] Then one computes the local capacity $c_w$
of $D_v\otimes_{F_v} K_w$ and
the (local) degree $d'_w$ of $D'_w$ by
$(\ref{eq:413})$ and $(\ref{eq:49})$, respectively.

\item [(iii)]  Using (\ref{eq:412}) we compute
the (global) degree $\delta'$ of $\Delta'$ and then
compute the (global) capacity $c$ of $\Delta\otimes K$ using (\ref{eq:48}).
\end{itemize} \

Theorem~\ref{29} states that $\Hom_{F}(K,A)\neq \empty$ if and only if
$k\mid nc$. This provides a simple numerical criterion  
how to check when there is an embedding $K\hookrightarrow A$.\\

Now we analyze the obstruction to the corresponding local-global
principle. Similar to (\ref{eq:36}) we define for each $v\in V^F$ a
finite set
\begin{equation}
  \label{eq:414}
\calE_v:=\calE_{F_v}(K_v,A_v)=\{(x_w)_{w|v}\,|\, x_w\in \bbN,\
\sum_{w|v} \ell_w x_w=n s_v \, \},
\end{equation}
where $\ell_w:=k_w/c_w$.

Theorem~\ref{33} (also with Theorem~\ref{27} (2)) 
states that there is an isomorphism
\[ e_v: A_v^\times \backslash \Hom_{F_v}^*(K_v,A_v)\isoto \calE_v. \]
If $\varphi_v\in \Hom_{F_v}^*(K_v,A_v)$, then it gives to a
decomposition of {\it non-zero} $D_v$-submodules $V_w$
\[ D_v^{ns_v}=\bigoplus_{w|v} V_w, \quad \dim_{D_v} V_w=n_w \]
with the property $k_w|n_w c_w$ or equivalently $\ell_w|n_w$. Then 
$e_v([\varphi_v])$ is given by the formula
\begin{equation}
    \label{eq:418}
  e_v([\varphi_v])=\left ( \frac{n_w}{\ell_w} \right )_{w|v}.  
\end{equation}
Let us suppose first that the set $\Hom_F(K, A)$ of embeddings from
$K$ into $A$ over $F$ is {\it non-empty}. 
For any element $\varphi$ in $\Hom_F(K,A)$, 
let $\varphi_v\in \Hom_{F_v}(K_v,A_v)$ be the extension 
of $\varphi$ by $F_v$-linearity, 
and let $[\varphi_v]$ be its equivalence class. 
Then one gets an element $\bfx_v\in \calE_v$ by
\[ \bfx_v:=e_v([\varphi_v])=(\bfx_w)_{w|v}. \]
Simple computation shows that
\begin{equation}
  \label{eq:415}
  \bfx_w:=\dim_{D_v} W_w/\ell_w=n s_v c_w/k,
\end{equation}
which is a positive integer.

Without knowing $\Hom_F(K,A)$ is non-empty, we still define an element
$\bfx_w\in \Q$ for each $w\in V^K$ by (\ref{eq:415}). Let $\bar
\bfx_w$ be the image of $\bfx_w$ in $\Q/\Z$. We associate to the
pair $(K,A)$ an element 
\[ \bar \bfx:=(\bar \bfx_w)_{w\in V^K}\in \bigoplus_{w\in V^K}
\Q/\Z. \]
The vanishing of the class $\bar \bfx$ is an obstruction for the set
$\Hom_F(K,A)$ be non-empty. Moreover, the following result
states that this is the only obstruction. 
\begin{thm}\label{46} Notations as above. We have
\[ \Hom^*_F(K,A)\neq \emptyset \iff \bar \bfx=0. \]
\end{thm}
\begin{proof}
  See \cite[Theorem 3.6]{shih-yang-yu}. \qed
\end{proof}

\subsection{Construction of examples}
\label{sec:44}

For a central simple algebra $A$ over a global field $F$ and a finite
field extension $K$ over $F$, we say that the Hasse principle  
for the pair $(K,A)$ holds if one has the equivalence of the conditions
\[ \Hom^*_F(K,A)\neq \emptyset \iff \Hom^*_{F_v}(K_v,A_v)\neq
\emptyset, \quad \forall\, v\in V^F. \]
In this subsection we shall construct a family of examples $(K,A)$ 
so that the Hasse principle for $(K,A)$ fails.

\begin{lemma}
  Let $S$ be a finite set of places of
  a global field $F$. Let $L_v$, for each $v\in S$, be any etale
  $F_v$-algebra of the same degree
  $[L_v:F_v]=d$. Then there exists a finite
  separable field extension $K$ of $F$ of degree $d$ such that
  $K\otimes_F F_v\simeq L_v$ 
  for all $v\in S$.  
\end{lemma}
\begin{proof}
This follows from the Hilbert irreducibility theorem and Krasner's
lemma; also see a proof in \cite[Lemma 3.2]{yu:const}. \qed  
\end{proof}

Let $K$ be any finite separable field extension of $F$ of degree
$k=[K:F]>1$, and let $\delta$ be a positive integer with more than one
prime divisors and divisible by $k$. Write 
$\delta=p_1^{n_1} \dots p_r^{n_r}$, 
$r\ge 2$ and $n_i\ge 1$, where each $p_i$ is a prime number. 

Assume that $k\le \delta/p_i^{n_i}$ for all $i$. We claim that there is
a central division algebra $\Delta$ over $F$ of degree $\delta$ so
that the Hasse principle for the pair $(K,\Delta)$ fails. 

Choose $2r$ places $v_1, v_1',\dots, v_r, v_r'$ of $F$ which split
completely in $K$. One has
\begin{equation}
  \label{eq:419}
  K_{v_i}=\prod_{j=1}^k K_{w_{ij}}, \quad K_{w_{ij}}\simeq F_{v_i}, \quad
  \text{ and }\quad K_{v'_i}=\prod_{j=1}^k 
  K_{w'_{ij}}, \quad K_{w'_{ij}}\simeq F_{v'_i}, \quad \forall\, i. 
\end{equation}
Choose a central division algebra $\Delta$ over $F$ with following
local invariants:
\begin{itemize}
\item $\inv_{v_i}(\Delta)=-\inv_{v'_i}(\Delta)=1/p_i^{n_i}$ for
  $i=1,\dots, r$, and
\item $\inv_v(\Delta)=0$ for other places $v$.
\end{itemize}
Then $\Delta$ has degree $\delta$. For all $i$, we have
\begin{equation}
  \label{eq:420}
  s_{v_i}=s_{v_i'}=\delta/p_i^{n_i}, \quad c_{w_{ij}}=1, \quad 
  \bfx_{w_{ij}}=s_{v_i}/k,
  \quad \ell_{w_{ij}}=1, 
\end{equation}
and 
\[ \calE_{v_i}\simeq \calE_{v_i'}=\left \{(x_j)\in \bbN^k\,  ; \,
\sum_{j=1}^k x_j=s_{v_i} \right \}. \]
Since $k\le s_{v_i}$, the sets $\calE_{v_i}$ and $\calE_{v'_i}$ are
non-empty; the remaining sets $\calE_v$ for unramified places $v$ are
also non-empty due to $k\,|\,\delta$. 

Suppose that the Hasse principle for $(K,\Delta)$ holds, then one has 
$\bfx_{w_{ij}}\in \bbN$ for all $i$ by Theorem~\ref{46}. This implies
that $k=1$ as $\gcd
\{s_{v_i}\}_i=1$,  a contradiction.

\section*{Acknowledgments}
  Part of this work was done during the author's stays at Tsinghua
  University in Beijing and at Universit\"at Duisburg-Essen. 
  He wishes to thank Linsheng Yin and U.~G\"ortz for
  their kind invitation and hospitality. 
  The author was partially supported by grants NSC
  97-2115-M-001-015-MY3 and AS-99-CDA-M01.

\end{document}